\documentclass[a4paper,10pt]{article}
\usepackage{amsmath,amsxtra,amssymb,latexsym, amscd,amsthm,authblk}
\usepackage[colorlinks=true]{hyperref}
\usepackage[margin=1in]{geometry}
\usepackage{algorithm}
\usepackage{comment}
\usepackage{amsmath,amssymb}
\usepackage{graphicx}
\usepackage{array}
\usepackage{xcolor}
\usepackage{enumerate}
\usepackage{graphicx}
\usepackage{cases}
\theoremstyle{plain}
\newtheorem{dl}{Theorem}[section]

\newtheorem{bd}{Lemma}[section]

\newtheorem{dn}{Definition}[section]

\newcommand{\ue}{u_\epsilon}
\newcommand{\ve}{v_\epsilon}

\newcommand{\uei}{u_{\epsilon_i}}
\newcommand{\vei}{v_{\epsilon_i}}

\title{Fast reaction limits and convergence rate for nonlinear bulk-surface reaction-diffusion systems modeling reversible chemical reactions}

\author[1]{The Tuan Hoang}
\author[2]{Nhu Phong Tham}
\author[3]{Bao Quoc Tang\footnote{Corresponding author.}}
\affil[1]{\small Institute of Mathematics, Vietnam Academy of Science and Technology, 18 Hoang Quoc Viet, Hanoi, 10307 Vietnam\break
	\href{mailto:httuan@math.ac.vn}{httuan@math.ac.vn}}
\affil[2]{\small Institute of Mathematics, Vietnam Academy of Science and Technology, 18 Hoang Quoc Viet, Hanoi, 10307 Vietnam\break   \href{mailto:phongtham.hn@gmail.com}{phongtham.hn@gmail.com}}
\affil[3]{\small Department of Mathematics and Scientific Computing, University of Graz, Graz, Austria\break  \href{mailto:quoc.tang@uni-graz.at}{quoc.tang@uni-graz.at}}
\date{}
\begin{document}
	\maketitle
	\begin{abstract}
		The fast reaction limit for a nonlinear bulk-surface reaction-diffusion system is investigated. This  system describes a reversible reaction with arbitrary stoichiometric coefficients, where one chemical is present in a bounded vessel $\Omega$ and the other chemical lies only on the boundary $\partial\Omega$ where the reaction takes place. In the limit as the reaction rate constant tends to infinity, we prove that the solution converges in $L^p(0,T;L^p(\Omega))$ to the solution of a heat equation with nonlinear dynamical boundary condition. This is obtained by showing a-priori estimates of solutions which are uniform in the reaction rate constants. In order to overcome the difficulty caused by the bulk-surface coupling, we consider the limit in suitable product spaces where the Aubin-Lions lemma is applicable. Moreover, in the case of equal stoichiometric coefficients, we obtain the convergence rate of the fast reaction limit by exploiting suitable estimates of the limiting system.
	\end{abstract}

	\section{Introduction}
	
	Many nature phenomena involve processes occurring simultaneously in a domain's interior (the bulk) and on its boundary (the surface), with intricate coupling between these two regions. From cell biological signalling proceeses (\cite{betschinger2003complex,fellner2018well,hausberg2018well}), to the ecological systems (\cite{berestycki2015effect,berestycki2013influence}) and in the area of material science (\cite{glitzky2013gradient,mielke2012thermomechanical}), the interaction between bulk and surface dynamics is fundamental to understanding these systems. The study of these problems gives rise to bulk-surface (or volume-surface) systems, which have been investigated extensively in the last decade. In this paper, we study the fast reaction limit of a nonlinear bulk-surface system modelling a reversible chemical reaction.
	
	\medskip
	To describe the problem under consideration, let $\Omega\subset\mathbb R^N$ be a bounded domain with smooth boundary $\Gamma:= \partial\Omega$, and $\alpha, \beta\ge 1$ be arbitrary constants. Let $\mathcal U$ be a bulk-chemical and $\mathcal V$ be a surface-chemical, and assume that they react on the surface $\Gamma$ according the following chemical reaction
	\begin{equation*}
		\alpha\mathcal{U} \underset{k}{\overset{k}{\leftrightarrows}} \beta\mathcal{V} 
	\end{equation*}
	where $k>0$ is the reaction rate constants. Assume that $\mathcal{U}$ diffuses in $\Omega$ and $\mathcal{V}$ diffuses along $\Gamma$, respectively, we can apply Fick's law for the diffusion and the mass action law for the reaction to obtain the following bulk-surface system
	\begin{equation}\label{sys}
		\begin{cases}
			\partial_t u - d_u\Delta u = 0, &x\in\Omega, t>0,\\
			\partial_{\mathbf n}u = -\alpha k(u^\alpha - v^\beta), &x\in\Gamma, t>0,\\
			\partial_tv - d_v\Delta_{\Gamma}v = \beta k(u^\alpha - v^\beta), &x\in\Gamma, t>0, \\
			u(x,0) = u_0(x), & x \in \Omega, \\
			v(x,0) = v_0(x), & x \in \Gamma.
		\end{cases}
	\end{equation}
	Here $u(x,t)$ and $v(x,t)$ represent for the concentrations of $\mathcal{U}$ on $\Omega$ and $\mathcal{V}$ on $\Gamma$, respectively, 
$\Delta_\Gamma$ is the Laplace--Beltrami operator on $\Gamma$ (see, e.g. \cite{jost2013riemannian}), $\mathbf{n}$ stands for the outward pointing unit normal vector field on $\Gamma$, and $d_u, d_v>0$ are diffusion coefficients. Moreover, the initial condition $(u_0,v_0)$ is assumed non-negative and bounded. From \eqref{sys}, we formally have the following mass conservation law
\begin{equation*}\label{eq:conservation}
	\alpha\int_\Omega u(x,t) dx + \beta\int_\Gamma v(x,t)dS = \alpha\int_\Omega u_0(x)dx + \beta\int_\Gamma v_0(x)dS = M_0 \geq 0
\end{equation*}
for all $t > 0$. 

\medskip
Bulk-surface systems similar to \eqref{sys} have recently attracted a lot of attention, since such systems arise naturally from many applications. Various aspects of these systems have been studied, including well-posedness (\cite{fellner2018well, hausberg2018well}), the stability of solutions (\cite{alfaro2026functional,stolerman2019stability}), numerical methods (\cite{alfaro2025long,egger2018analysis}), rigorous derivation (\cite{li2021bulk, bobrowski2025existence}), and many others. On the other hand, fast reaction limits for reaction-diffusion systems have been investigated extensively in recent years. These problems occur when reaction rate constants become significantly faster than other parameters, such as diffusion coefficients, and the limit system can be of various type, ranging from nonlinear diffusion system \cite{bothe2003reaction}, cross-diffusion system \cite{bothe2012cross,daus2020cross}, free boundary problem \cite{bothe2011instantaneous,murakawa2011fast} or reaction-diffusion system with different kinetics \cite{tang2024rigorous}. We also refer the interested reader to the \cite{hilhorst2023lecture} and references therein. 

\medskip
The study of singular limits for bulk-surface systems has been still limited. There are several works considering the fast diffusion limits such as \cite{hausberg2018well,niethammer2020bulk,logioti2021parabolic}, where the limit is an obstacle or free boundary problem. Up to our knowledge, the only works considering fast reaction limit for bulk-surface system are \cite{fellner2016quasi} and  \cite{henneke2016fast}, where the latter considered system \eqref{sys} in the linear setting, i.e. $\alpha = \beta = 1$. In this paper, we extend this result to the nonlinear case, that is $\alpha, \beta \ge 1$ arbitrary, which has been left as an open problem in \cite{henneke2016fast}. In order to go more into details, we first look at the formal limit. Denote by $(u_k,v_k)$ the solution to \eqref{sys} corresponding to $k>0$. Formally, as $k \to \infty$, we expect that $u_k \to u$ in $\Omega\times(0,T)$, while $v_k \to v$ and $u_k^\alpha - v_k^\beta \to 0$ in $\Gamma\times (0,T)$. Thus, we expect the limit system to read as
\begin{equation}\label{limit_system}
	\begin{cases}
		\partial_t w - d_u\Delta w = 0, &x\in\Omega, t>0,\\
		d_u\nabla w \cdot {\mathbf n} = -\frac{\alpha}{\beta}\big(\partial_t (w^{\alpha/\beta}) - d_v\Delta_\Gamma (w^{\alpha/\beta})\big), &x\in\Gamma, t>0,\\
		w(x,0) = u_0(x), &x\in\Omega,\\
		w|_{\Gamma}(x,0) = v_0^{\beta/\alpha}(x), &x\in\Gamma.
	\end{cases}
\end{equation}
This is a parabolic equation with {\it nonlinear} dynamical boundary condition, which, up to our knowledge, has not been studied in the literature, except for the case $\alpha = \beta$, see e.g.  (\cite{VAZQUEZ20112143, escher1993quasilinear}). 

\medskip
Our first main result is the strong convergence of the fast reaction limit. Thanks to the reversibility of the reaction, the system \eqref{sys} possesses an entropy functional which is decreasing in time. This helps to get certain a-priori estimates of the solutions which are uniformly in $k$, and moreover it shows that $u_k^\alpha - v_k^{\beta} \to 0$ as $k\to \infty$. In order to get the strong convegence of $u_k$ and $v_k$, we need certain estimates of the time derivative. In \cite{bothe2003reaction}, when two equations are both posed in the bulk, the authors therein sum them up and show the strong convergence of the sum of the solutions, as consequently for each separately.  
Due to the bulk-surface coupling, this idea of summing up the equations does not work. For the linear case, the authors in \cite{henneke2016fast} utilized the linearity of the system and Ball's energy method to prove the convergence of fast reaction limit without having to invoke the Aubin-Lions lemma. This technique, however, seems not to be applicable to the nonlinear system \eqref{sys}. In order to overcome this issue, our main idea is to consider the solutions $(u_k,v_k)$ in a product space, and show that the time derivative of this pair is bounded in $L^2(0,T;Z^*)$ where $Z^*$ is the dual space of $Z:= \{(\phi,\phi|_{\Gamma}) \text{ for all } \phi\in H^1(\Omega) \text{ such that } \phi|_{\Gamma}\in H^1(\Gamma)\}$. Our second main result is the convergence rate of the fast reaction limits. More precisely, we consider the case of equal stoichiometric coefficients $\alpha = \beta$, and utilize certain properties of the limit system to estimate the difference $u_k - u$ and $v_k - v$ directly to get the convergence rate
\begin{equation*}
	\|u_k(t) - w(t)\|_{L^2(\Omega)} + \|v_k(t) - w|_{\Gamma}(t)\|_{L^2(\Gamma)} \le Ck^{-1/2}, \quad \forall t\in (0,T).
\end{equation*}

\medskip
This paper will be structured as follows. In Section \ref{sec2}, we will consider the fast reaction limits, where we first prove the uniqueness and existence of a solution for the system \eqref{sys}, then we will show that as the reaction rate constant tends infinity, the solution of \eqref{sys} will converge to that of the limit system. In Section 3, we will start with some properties of the limit system and quantify the convergence rate.

\medskip
For the rest of this paper, we will use the notation $\langle \cdot, \cdot \rangle_{\Omega}$ and $\langle \cdot, \cdot \rangle_{\Gamma}$ for the duality between $H^{-1}(\Omega)$ and $H^1(\Omega)$, $H^{-1}(\Gamma)$ and $H^1(\Gamma)$, respectively. Besides, we will change the variable $k = 1/\epsilon$, which means that the question becomes considering the limit $\epsilon \to 0^+$.
\section{Fast reaction limit}\label{sec2}
\subsection{Existence and uniqueness of solution}
The existence of a unique global solution to \eqref{sys} is a special case of \cite{fellner2018well}. Therein the reaction rate constants can be of general form, i.e. $k = k(x,t)$, and the authors used a compliciated mechanism of upper and lower solutions to get the results. In this paper, we provide a more elementary and direct proof for the existence of \eqref{sys}.
First, we recall the definition of a weak solution to \eqref{sys}.
\begin{dn}\label{def:weak_sol}
	A pair of functions $(u,v)$ is called a weak solution of \eqref{sys} on $(0,T)$ (with given $T > 0$) if 
	\begin{align}\label{def:regualar}
		&u \in C([0,T]; L^2(\Omega)), \quad u \in L^\infty(0,T;L^\infty(\Omega)) \cap L^2(0,T;H^1(\Omega)), \\ 
		&v \in C([0,T]; L^2(\Gamma)), \quad v \in L^\infty(0,T;L^\infty(\Gamma)) \cap L^2(0,T;H^1(\Gamma)),
	\end{align}
	and the following weak formulation holds
	\begin{equation}\label{wf1}
		\left\{
		\begin{aligned}
			&\int_0^T \int_\Omega (-u\varphi_t + d_u \nabla u \nabla \varphi)dxdt = \int_\Omega u_0\varphi(0)dx - \frac{\alpha}{\epsilon} \int_0^T\int_\Gamma (u^\alpha - v^\beta)\varphi dSdt \\
			&\int_0^T \int_\Gamma (-v\psi_t + d_v \nabla_\Gamma v \nabla_\Gamma \psi)dSdt = \int_\Gamma v_0\psi(0)dS + \frac{\beta}{\epsilon} \int_0^T\int_\Gamma (u^\alpha - v^\beta)\psi dSdt
		\end{aligned}
		\right.
	\end{equation}
	for all test functions $\varphi \in C^1([0,T];L^2(\Omega))$ $\cap$ $L^2(0,T;H^1(\Omega))$ and $\psi \in C^1([0,T]$ $;L^2(\Gamma)) \, \cap \, L^2(0,T;H^1(\Gamma))$, with $\varphi(T)  \equiv 0, \psi(T) \equiv 0$.
\end{dn}
\begin{dl}\label{exist_sol}
	For any non-negative and bounded initial data $(u_0,v_0)\in L^{\infty}(\Omega)\times L^{\infty}(\Gamma)$, there exists a unique global non-negative weak solution to \eqref{sys}.
\end{dl}

Let $\delta>0$ and consider the following approximate system
\begin{equation}\label{approximation_system}
	\begin{cases}
		\partial_t u_\delta - d_u\Delta u_\delta = 0, &x\in\Omega, t >0,\\
		\partial_{\mathbf n} u_{\delta} = -\frac{\alpha}{\epsilon}  F_\delta(u_\delta,v_\delta), &x\in\Gamma, t>0,\\
		\partial_t v_\delta - d_v\Delta_\Gamma v_\delta = \frac{\beta}{\epsilon} F_\delta(u_\delta, v_\delta), &x\in \Gamma, t>0,\\
		u_\delta(x,0) = u_0(x), &x\in \Omega, \\
		v_\delta(x,0) = v_0(x), &x\in \Gamma,
	\end{cases}
\end{equation}
where
\begin{equation*}
	F_\delta(u,v):= \frac{u^\alpha - v^\beta}{1 + \delta|u^\alpha - v^\beta|}.
\end{equation*}
Remark that $F_\delta(u,v)$ is locally Lipschitz continuous in $\mathbb R^2$ and for each $\delta > 0$, we have
\begin{equation*}
	|F_\delta(u,v)| \le \frac{1}{\delta}\quad  \forall (u,v) \in \mathbb R_+^2.
\end{equation*}
To show the existence of the approximate solution, we use Galerkin's method (see e.g. \cite[Theorem 7.1]{evans}).
\begin{bd}\label{lm:exist_approx_sol}
	The system \eqref{approximation_system} has unique (weak) solution $(u_\delta,v_\delta)$ for each parameter $\delta > 0$.
\end{bd}
\begin{proof}
	Similarly to Definition \ref{def:weak_sol}, $(u_\delta, v_\delta)$ is a weak solution to \eqref{approximation_system}  if 
	\begin{align}\label{def:regualar_approximation}
		&u_\delta \in C([0,T]; L^2(\Omega)), \quad u_\delta \in L^\infty(0,T;L^\infty(\Omega)) \cap L^2(0,T;H^1(\Omega)), \\ 
		&v_\delta \in C([0,T]; L^2(\Gamma)), \quad v_\delta \in L^\infty(0,T;L^\infty(\Gamma)) \cap L^2(0,T;H^1(\Gamma)),
	\end{align}
	and the following weak formulation holds
	\begin{equation}\label{wf_approximation}
		\left\{
		\begin{aligned}
			&\int_0^T \int_\Omega (-u_\delta\varphi_t + d_u \nabla u_\delta \nabla \varphi)dxdt = \int_\Omega u_0\varphi(0)dx - \frac{\alpha}{\epsilon} \int_0^T\int_\Gamma F_\delta(u_\delta,v_\delta)\varphi dSdt \\
			&\int_0^T \int_\Gamma (-v_\delta\psi_t + d_v \nabla_\Gamma v_\delta \nabla_\Gamma \psi)dSdt = \int_\Gamma v_0\psi(0)dS + \frac{\beta}{\epsilon} \int_0^T\int_\Gamma F_\delta(u_\delta,v_\delta)\psi dSdt
		\end{aligned}
		\right.
	\end{equation}
	for all test functions $\varphi \in C^1([0,T];L^2(\Omega))$ $\cap$ $L^2(0,T;H^1(\Omega))$ and $\psi \in C^1([0,T]$ $;L^2(\Gamma)) \, \cap \, L^2(0,T;H^1(\Gamma))$, with $\varphi(T) \equiv 0, \psi(T) \equiv 0$.\\
	
	\medskip
	Choosing an orthogonal basis $(\varphi_k, \psi_k)_{k=1}^\infty$ of  $L^2(\Omega) \times L^2(\Gamma)$ (for example, $\varphi_k, \psi_k$ are the eigenfunctions (in the order that the eigenvalues are increasing) of the Laplace operator on $\Omega$ with Neumann boundary condition and Laplace--Beltrami operator on $\Gamma$ (see, e.g., \cite{rosenberg1997laplacian}), respectively). For each fixed positive integer $m$, we seek the coefficents $(a_{i,m}(t), b_{i,m}(t))$ of $(u_{\delta,m}, v_{\delta,m})$, the projection of $(u_\delta(t), v_\delta(t))$ on finite-dimenional subspace spanned by $(\varphi_i, \psi_i)_{i=1}^m$:
	\begin{equation*}
		\begin{cases}
			u_{\delta,m}(\cdot,t) = \sum_{i=1}^m \varphi_i a_{i,m}(t), \\
			v_{\delta,m}(\cdot,t) = \sum_{i=1}^m \psi_i b_{i,m}(t),
		\end{cases}
	\end{equation*}
	that satisfies the following finite systems for a.e. $t \in [0,T]$ and each $j = 1,\dots,m$
	\begin{equation*}
		\left\{
		\begin{aligned}
			&\langle \partial_t u_{\delta,m},\varphi_j\rangle_\Omega + d_u\int_\Omega \nabla u_{\delta,m} \nabla \varphi_j dx = \int_\Gamma-\frac{\alpha}{\epsilon}F_{\delta}(u_{\delta,m}v_{\delta,m}) \varphi_j dS  \\
			&\langle \partial_t v_{\delta,m},\psi_j\rangle_\Gamma + d_v\int_\Gamma \nabla_\Gamma v_{\delta,m} \nabla_\Gamma \psi_j dS = \int_\Gamma \frac{\beta}{\epsilon}F_{\delta}(u_{\delta,m},v_{\delta,m}) \psi_j dS 
		\end{aligned}
		\right.
	\end{equation*}
	and for the inital condition, we require
	\begin{equation*}
		a_{i,m}(0) = \int_\Omega u_0 \varphi_i dx, \quad b_{i,m}(0) = \int_\Gamma v_0 \psi_i dS,\quad \forall i = 1,..,m.
	\end{equation*}
	Thanks to the local Lipschitz continuity and boundedness of $F_\delta$, this system of ODEs has a unique global solution. We now show the boundedness of $\{u_{\delta,m
	},v_{\delta,m}\}_{m=1}^\infty$ uniformly in $m$. From the finite systems, we get
	\begin{align}
		\frac{1}{2}\frac{d}{dt}\left(\|u_{\delta,m}\|^2_{L^2(\Omega)} + \|v_{\delta,m}\|^2_{L^2(\Gamma)}\right) + d_u \|\nabla u_{\delta,m}\|^2_{L^2(\Omega)} + d_v \|\nabla_\Gamma v_{\delta,m}\|^2_{L^2(\Gamma)} \nonumber \\ \leq \frac{|\beta - \alpha|}{\epsilon \,\delta}\int_\Gamma |u_{\delta,m}|+ |v_{\delta,m}| \,dS. \label{eq:1.1.1}
	\end{align}
	The right hand side can be estimated by
	\begin{align}
		\frac{|\beta - \alpha|}{\epsilon \,\delta}\int_\Gamma |u_{\delta,m}|+ |v_{\delta,m}| \,dS &\leq C(\|u_{\delta,m}\|_{L^2(\Gamma)} + \|v_{\delta,m}\|_{L^2(\Gamma)}) \nonumber \\&\leq C(\|u_{\delta,m}\|^2_{L^2(\Omega)} + \|v_{\delta,m}\|^2_{L^2(\Gamma)}) + C \nonumber \\&+ \frac{d_u}{2} \|\nabla u_{\delta,m}\|^2_{L^2(\Omega)}, \label{eq:1.1.2}
	\end{align}
	where the second one is the consequence of modified trace Theorem : for any $\kappa>0$, there exists a constant $C_\kappa$ such that $\|u\|_{L^2(\Gamma)} \leq \kappa\|\nabla u\|_{L^2(\Omega)} + C_\kappa \|u\|_{L^2(\Omega)}$. Combining \eqref{eq:1.1.1} and \eqref{eq:1.1.2}, we get
	\begin{equation*}
		\frac{1}{2}\frac{d}{dt}\left(\|u_{\delta,m}\|^2_{L^2(\Omega)} +  \|v_{\delta,m}\|^2_{L^2(\Gamma)}\right) \leq C (\|u_{\delta,m}\|^2_{L^2(\Omega)} +  \|v_{\delta,m}\|^2_{L^2(\Gamma)}) + C
	\end{equation*}
	Then, by Gronwall's inequality, we have the following boundedness for a.e. $ t \in [0,T]$ $$\|u_{\delta,m}(t)\|^2_{L^2(\Omega)} +  \|v_{\delta,m}(t)\|^2_{L^2(\Gamma)} \leq C.$$  
	Next, integrating \eqref{eq:1.1.1} on $(0,T)$, and using the fact that $\|u_{\delta,m}\|^2_{L^2(\Omega)} +  \|v_{\delta,m}\|^2_{L^2(\Gamma)}$ are bounded, we also get $$\|\nabla u_{\delta,m}\|^2_{L^2(0,T;L^2(\Omega))} + \|\nabla v_{\delta,m}\|^2_{L^2(0,T;L^2(\Gamma))} \leq C,$$ where $C$ does not depend on $m$. 
	In other word, the sequence $\{(u_{\delta,m}, v_{\delta,m})\}_{m=1}^\infty$ is bounded uniformly in $L^2(0,T;H^1(\Omega) \times H^1(\Gamma))$. 
	From to the finite dimensional system, we have for $\varphi \in L^2(0,T;H^1(\Omega))$
	\begin{align*}
		\int_0^T \langle \partial_t u_{\delta,m}, \varphi \rangle dt &\leq d_u \|\nabla u_{\delta,m}\|_{L^2(0,T;L^2(\Omega))}\|\nabla \varphi\|_{L^2(0,T;L^2(\Omega))} + \frac{\alpha}{\epsilon\, \delta}\int_0^T\int_\Gamma \varphi dSdt \\ &\leq C\|\varphi\|_{L^2(0,T;H^1(\Omega))},
	\end{align*}
	where $C$ does not depend on $\varphi$ and $m$. So, we have $\{\partial_t u_{\delta,m}\}_{m=1}^\infty$ is bounded uniformly in $L^2(0,T;(H^{1}(\Omega)^*)$. Then, by applying Aubin--Lions lemma (see, e.g. \cite{moussa2016some}) there exists a subsequence of $\{u_{\delta,m}\}$ converge strongly to some $u_\delta$ in $L^2(0,T;L^2(\Omega))$ as $m \to \infty$. Moreover, $u_\delta \in L^2(0,T;H^1(\Omega))$. Similarly, there also exists a subsequence of $\{v_{\delta,m}\}$ converge strongly to some $v_\delta$ in $L^2(0,T;L^2(\Gamma))$ and $v_\delta \in L^2(0,T;H^1(\Gamma))$. This is sufficient to pass to the limit as $m\to\infty$ the finite system to obtain that the limit $(u_\delta,v_\delta)$ satisfies the weak formulation \eqref{wf_approximation}, which leads to the existence of solution of \eqref{approximation_system}. Since $|F_\delta(u,v)|\le 1/\delta$ in $\mathbb R^2$, the boundedness of $(u_\delta,v_\delta)$ follows from standard arguments for parabolic equations. Combining this boundedness with the local Lipschitz continuity of $F_\delta$, we see that this solution is also unique. 
	
	\medskip
	Next, we prove the non-negativity of the solution by considering the auxillary system 
	\begin{equation}\label{aux_system}
		\begin{cases}
			\partial_t u_\delta - d_u\Delta u_\delta = 0, &x\in\Omega, t>0\\
			\partial_{{\mathbf n}}u_{\delta} = -\frac{\alpha}{\epsilon}  F_\delta(u^+_\delta,v^+_\delta), &x\in\Gamma, t>0\\
			\partial_t v_\delta - d_v\Delta_\Gamma v_\delta = \frac{\beta}{\epsilon}F_\delta(u^+_\delta, v^+_\delta), &x\in \Gamma, t>0\\
			u_\delta(0) = u_0, & x\in \Omega,\\
			v_\delta(0) = v_0 & x \in \Gamma,
		\end{cases}
	\end{equation}
	where $u^+ = \max\{0,u\}$ and $u^- = \max\{0,-u\}$. It is observed that
	\begin{equation*}
		\int_\Gamma F_{\delta}(u_\delta^+,v_\delta^+)u_\delta^- dS = \int_\Gamma \frac{(u_\delta^+)^\alpha - (v_\delta^+)^\beta}{1 + \delta |(u_\delta^+)^\alpha - (v_\delta^+)^\beta|}u_\delta^- dS \leq 0
	\end{equation*}
	since $u_\delta^+ u_\delta^- = 0$ and $u_\delta^-, v_\delta^+ \geq 0$. Similarly, we have 
	\begin{equation*}
		\int_\Gamma F_{\delta}(u_\delta^+,v_\delta^+)v_\delta^- dS \geq 0
	\end{equation*}
	Multiply the first and the third equations of \eqref{aux_system} with $(-u^-_\delta, -v^-_\delta)$ and applying integration by parts formula, we obtain
	\begin{align*}
		\frac{1}{2}\frac{d}{dt} (\|u_\delta^-\|^2_{L^2(\Omega)} + \|v_\delta^-\|^2_{L^2(\Gamma)}) + d_u \|\nabla u_\delta^-\|_{L^2(\Omega)}^2 +  d_v \|\nabla v_\delta^-\|^2_{L^2(\Gamma)} \\ = \int_{\Gamma} \frac{\alpha}{\epsilon}F_\delta(u^+_\delta, v^+_\delta)u^-_\delta - \frac{\beta}{\epsilon}F_\delta(u^+_\delta, v^+_\delta)v^-_\delta dS.
	\end{align*}
	The right hand side is non-negative, which leads to
	\begin{equation*}
		\frac{d}{dt} (\|u_\delta^-\|^2_{L^2(\Omega)} + \|v_\delta^-\|^2_{L^2(\Gamma)}) \leq 0.
	\end{equation*}
	Combine with the fact that $(u_\delta(0),v_\delta(0))$ is non-negative, we have the identity for a.e. $t \in [0,T]$
	\begin{equation*}
		\|u_\delta^-(t)\|^2_{L^2(\Omega)}  = \|v_\delta^-(t)\|^2_{L^2(\Gamma)} = 0.
	\end{equation*}
	So, the system \eqref{aux_system} has non-negative solution, which also is a solution of \eqref{approximation_system}. Due to the uniqueness of solutions to \eqref{approximation_system}, $(u_\delta, v_\delta)$ is non-negative.
\end{proof}
To compelete the proof of Theorem \ref{exist_sol}, we need to pass the limit as $\delta \to 0$ to show that the limit of the sequence $(u_\delta, v_\delta)$ is the solution of \eqref{sys}. We start with some boundedness uniformly in $\delta$. We will introduce a class of non-increasing entropy functions of $L^p$-type.
\begin{bd}\label{lm:ub_approx}
	There exists a positive constant $M$, which is independent of $\delta$, such that 
	\begin{equation*}
		\|u_\delta\|_{L^{\infty}(0,T;L^{\infty}(\Omega))} + \|u_\delta\|_{L^{\infty}(0,T;L^{\infty}(\Gamma))} \le M.
	\end{equation*}
\end{bd}
\begin{proof}
	Let $p\in \mathbb N \backslash\{0\}$, we consider the entropy function
	\begin{equation}\label{entropy_func1.1}
		E_p[u_\delta,v_\delta](t) := \frac{1}{p\alpha^2 + \alpha}\int_\Omega u_\delta^{p\alpha + 1} dx + \frac{1}{p\beta^2 + \beta}\int_\Gamma v\delta^{p\beta + 1} dS,
	\end{equation}
	From the system \eqref{approximation_system}, we can show that the time-derivative of \eqref{entropy_func1.1} is non-positive for each $p$. Indeed,
	\begin{align*}
		\frac{d}{dt} E_p(t)
		=& -\frac{d_u}{\alpha}\int_\Omega p u_\delta^{p\alpha - 1}|\nabla u_\delta|^2 dx -\frac{d_v}{\beta}\int_\Gamma p v_\delta^{p\beta - 1}|\nabla_\Gamma v_\delta|^2 dS \\
		&+\frac{1}{\epsilon}\int_\Gamma \frac{(u_\delta^\alpha - v_\delta^\beta)((v_\delta^{\beta})^p - (u_\delta^{\alpha})^p)}{1 + \delta |u_\delta^\alpha - v_\delta^\beta|}dS \leq 0
	\end{align*}
	for all $p \in \mathbb{Z}^+$. Thus, for any $t_0 \in [0,T]$, we have
	\begin{align}\label{ieq1.3.1}
		\frac{1}{p\alpha^2 + \alpha}\int_\Omega u_\delta^{p\alpha + 1}(t_0) dx + &\frac{1}{p\beta^2 + \beta}\int_\Gamma v_\delta^{p\beta + 1}(t_0) dS \leq \nonumber \\ &\frac{1}{p\alpha^2 + \alpha}\int_\Omega u_0^{p\alpha + 1} dx + \frac{1}{p\beta^2 + \beta}\int_\Gamma v_0^{p\beta + 1} dS.
	\end{align}
	Therefore,
	\begin{equation*}
		\frac{1}{2\alpha^2}\int_\Omega u_\delta^{p\alpha + 1}(t_0) dx + \frac{1}{2\beta^2 }\int_\Gamma v_\delta^{p\beta + 1}(t_0) dS \leq\int_\Omega u_0^{p\alpha + 1} dx + \int_\Gamma v_0^{p\beta + 1} dS.
	\end{equation*}
	By taking the root of $p$ on both sides, we obtain
	\begin{equation}\label{ieq1.3.2}
		\left( \frac{1}{2\alpha^2}\int_\Omega u_\delta^{p\alpha + 1}(t_0) dx + \frac{1}{2\beta^2}\int_\Gamma v_\delta^{p\beta + 1}(t_0) dS \right)^{1/p}
		\leq  \left( \int_\Omega u_0^{p\alpha + 1} dx + \int_\Gamma v_0^{p\beta + 1} dS \right)^{1/p}.
	\end{equation}
	Due to the fact that $u_0,v_0$ are bounded, the left hand side is bounded uniformly to $p$ and $\delta$. 
	Letting $p\to \infty$ yields
	\begin{equation*}
		\|u_\delta(t_0)\|_{L^{\infty}(\Omega)} + \|v_\delta(t_0)\|_{L^{\infty}(\Gamma)} \le C
	\end{equation*}
	where $C$ is independent of $\delta>0$.
\end{proof}
From the uniform boundedness of $u_\delta$ on the domain, we can prove that its trace is also bounded.
\begin{bd}\label{lm: trace_esti}
	It holds that $\|u_\delta\|_{L^{\infty}(0,T;L^{\infty}(\Gamma))} \leq M$, for a constant $M$ that does not depend on $\delta$.
\end{bd}
\begin{proof}
	This lemma is a direct consequence of the previous lemma and the following estimate (see \cite[Exercise 4.1]{tröltzsch2010optimal})
	\begin{equation*}
		\|u\|_{L^{\infty}(\Gamma)} \leq \|u\|_{L^\infty(\Omega)} \quad \forall u \in L^\infty(\Omega) \cap H^1(\Omega).
	\end{equation*}
\end{proof}
Next, we will show that $(u_\delta,v_\delta)$ is bounded uniformly in $L^2(0,T;H^1(\Omega) \times H^1(\Gamma))$.
\begin{bd}\label{lm:gradient_approx}
	The approximate solution $(u_\delta, v_\delta)$ of \eqref{approximation_system}  is bounded uniformly in $L^2(0,T; H^1(\Omega) \times H^1(\Gamma))$.
\end{bd}
\begin{proof}
	We consider the logarithm entropy function (see \cite{fellner2018well}):
	\begin{equation}\label{log_entropy1.1}
		E[u_\delta,v_\delta](t) = \int_\Omega u_\delta(\log u_\delta - 1)dx + \int_\Gamma v_\delta (\log v_\delta - 1)dS,
	\end{equation}
	where $\log(x)$ is the natural logarithm of the positive number $x$. Formally, the dissipation of entropy function is calculated $D(t) = \frac{-d}{dt} E(t) = \langle -\partial_t u_\delta, \log(u_\delta) \rangle_\Omega + \langle -\partial_t v_\delta, \log(v_\delta) \rangle_\Gamma$ and then choose $\log(u_\delta)$ (and $\log(v_\delta)$, respectively) as the test functions. However, $\log(u_\delta)$ and $\log(v_\delta)$ do not belong to $L^2(\Omega)$ and $L^2(\Gamma)$ since we do not have yet positive lower bounds of $u_\delta$ and $v_\delta$, respectively. To overcome this, we replace $u_\delta$ (and $v_\delta$) by $u_\delta + \lambda$ (and $v_\delta + \lambda$), with $\lambda$ is a small positive parameter and taking the limit $\lambda \to 0^+$. We introduce the modified entropy function
	\begin{equation}\label{log_entropy1.2}
		E_\lambda[u_\delta,v_\delta](t) = \int_\Omega (u_\delta + \lambda)(\log (u_\delta+\lambda) - 1)dx \nonumber + \int_\Gamma (v_\delta + \lambda)(\log (v_\delta + \lambda) - 1)dS.
	\end{equation}
	We calculate the time-derivative entropy function along trajectories
	\begin{align*}
		D_\lambda(t) &= - \frac{d}{dt}E_\lambda(t)\\
		&= d_u\int_\Omega \nabla u_\delta  \nabla (\log(u_\delta + \lambda))dx + \frac{\alpha}{\delta}\int_\Gamma \frac{u_\delta^\alpha - v_\delta^\beta}{1 + \delta |u^\alpha_\delta - v^\beta_\delta|}\log(u_\delta + \lambda)dS \\
		&\quad +d_v\int_\Gamma \nabla_\Gamma v_\delta  \nabla_\Gamma (\log(v_\delta + \lambda))dS - \frac{\beta}{\delta}\int_\Gamma \frac{u_\delta^\alpha - v_\delta^\beta}{1 + \delta |u^\alpha_\delta - v^\beta_\delta|}\log(v_\delta + \lambda)dS \\
		&= \, d_u\int_\Omega \frac{|\nabla u_\delta|^2 }{u_\delta + \lambda}dx + d_v\int_\Gamma \frac{|\nabla_\Gamma v_\delta|^2}{v_\delta + \lambda}dS + \frac{1}{\delta}\int_\Gamma \frac{u_\delta ^\alpha - v_\delta ^\beta}{1+\delta|u^\alpha_\delta - v^\beta_\delta|}\log\frac{(u_\delta + \lambda)^\alpha}{(v_\delta + \lambda)^\beta}dS .
	\end{align*}
	Then, by integrating $D_\lambda(t)$ from $0$ to $T$, we get
	\begin{align*}
		E_\lambda(0)-E_\lambda(T)&=  \int_0^T D_\lambda(t)dt\\
		&= \int_0^T \left( d_u\int_\Omega \frac{|\nabla u_\delta|^2 }{u_\delta + \lambda}dx + d_v\int_\Gamma \frac{|\nabla_\Gamma v_\delta|^2}{v_\delta + \lambda}dS \right)dt
		\\ &+ \frac{1}{\delta}\int_0^T  \int_\Gamma \frac{u_\delta ^\alpha - v_\delta ^\beta}{1+\delta|u_\delta^\alpha - v_\delta^\beta|}\log\frac{(u_\delta + \lambda)^\alpha}{(v_\delta + \lambda)^\beta}dS dt.
	\end{align*}
	We can rewrite the above equation to:
	\begin{align*}
		\int_0^T \left( d_u\int_\Omega \frac{|\nabla u_\delta|^2 }{u_\delta + \lambda}dx + d_v\int_\Gamma \frac{|\nabla_\Gamma v_\delta|^2}{v_\delta + \lambda}dS \right)&dt  = E_\lambda(0) - E_\lambda(T) \\ + \frac{1}{\epsilon}\int_0 ^T \int_\Gamma &\frac{v_\delta ^\beta - u_\delta^\alpha}{1+\delta|u_\delta^\alpha - v^\beta_\delta|}\log\frac{(u_\delta + \lambda)^\alpha}{(v_\delta + \lambda)^\beta}dS
	\end{align*}
	Since the initial condition $(u_0,v_0) \in L^{\infty}(\Omega) \times L^\infty(\Gamma)$, the term $E_\lambda(0)$ can be bounded uniformly (by choosing $\lambda < 1$). On the other hand, $E_\lambda (T)$ has a lower bound without depending on $\lambda$, thanks to the boundedness of $(u_\delta,v_\delta)$. Thus, we have
	\begin{align}\label{est1.4.1}
		d_u\int_0^T \int_\Omega \frac{|\nabla u_\delta|^2 }{u_\delta + \lambda}dxdt +& d_v\int_0^T\int_\Gamma \frac{|\nabla_\Gamma v_\delta|^2}{v_\delta + \lambda}dSdt \nonumber \\ &\leq C + \frac{1}{\epsilon}\int_0 ^T \int_\Gamma \frac{v_\delta ^\beta - u_\delta^\alpha}{1+\delta|u^\alpha_\delta - v^\beta_\delta|}\log\frac{(u_\delta + \lambda)^\alpha}{(v_\delta + \lambda)^\beta}dS,
	\end{align}
	with $C$ is a constant that does not depend on $\lambda$ and also $\delta$. By using the Monotone Convergence Theorem,
	\begin{align*}
		\lim_{\lambda \to 0^+} \left( d_u\int_0^T \int_\Omega \frac{|\nabla u_\delta|^2 }{u_\delta + \lambda}dx + d_v\int_0^T\int_\Gamma \frac{|\nabla_\Gamma v_\delta|^2}{v_\delta + \lambda}dS \,dt \right)&   \\ = d_u\int_0^T \int_\Omega \frac{|\nabla u_\delta|^2 }{u_\delta}dxdt + d_v& \int_0^T\int_\Gamma \frac{|\nabla_\Gamma v_\delta|^2}{v_\delta}dSdt .
	\end{align*}
	For the logarithmic term, we rewrite it as
	\begin{align*}
		&\int_0^T\int_\Gamma \frac{v_\delta ^\beta - u_\delta^\alpha}{1+\delta|u^\alpha_\delta - v^\beta_\delta|}\log\frac{(u_\delta + \lambda)^\alpha}{(v_\delta + \lambda)^\beta}dSdt\\
		&= \int_0^T\int_\Gamma \frac{(v_\delta+\lambda)^\beta - (u_\delta+\lambda)^\alpha}{1+\delta|u^\alpha_\delta - v^\beta_\delta|}\log\frac{(u_\delta + \lambda)^\alpha}{(v_\delta + \lambda)^\beta}dSdt\\
		&\quad + \int_0^T\int_{\Gamma} \frac{(u_\delta+\lambda)^\alpha - u_\delta^\alpha - ((v_\delta+\lambda)^\beta - v_\delta^\beta)}{1+\delta|u_\delta^\alpha - v_\delta^\beta|}\log\frac{(u_\delta+\lambda)^\alpha}{(v_\delta+\lambda)^\beta}dSdt\\
		&=: (I_\lambda) + (II_\lambda).
	\end{align*}
	The convergence
	\begin{equation*}
		\limsup_{\lambda\to 0}(I_\lambda) \le -\int_0^T\int_{\Gamma}(u_\delta^\alpha - v_\delta^\beta)\log\frac{u_\delta^\alpha}{v_\delta^\beta}dSdt
	\end{equation*}
	follows from Fatou's lemma and the elementary inequality $-(u^\alpha - v^\beta)\log(u^\alpha/v^\beta) \le 0$. For $(II_\lambda)$, we use the $L^\infty$-boundedness of $u_\delta, v_\delta$, the estimate $|(u_\delta+\lambda)^\alpha - u_\delta^\alpha| + |(v_\delta+\lambda)^\beta - v_\delta^\beta| \le C\lambda$ and the Dominated Convergence Theorem to get
	\begin{equation*}
		\limsup_{\lambda \to 0}|(II_\lambda)| = 0.
	\end{equation*}
	Combining these estimates and take the limit $\lambda \to 0^+$ in \eqref{est1.4.1}, we obtain the uniform estimation
	\begin{equation*}
		d_u\int_0^T\int_\Omega \frac{|\nabla u_\delta|^2 }{u_\delta}dxdt +
		d_v\int_0^T\int_\Gamma \frac{|\nabla_\Gamma v_\delta|^2}{v_\delta}dSdt\leq C,
	\end{equation*}
	where $C$ is a constant that does not depend on $\lambda$. Finally, from the uniform boundedness of $(u_\delta,v_\delta)$ in $L^\infty(0,T;L^\infty(\Omega) \times L^\infty(\Gamma))$, we can show that the gradient is also bounded uniformly
	\begin{align*}
		\|\nabla u_\delta\|^2_{L^2(0,T;L^2(\Omega))} &= \int_0^T \int_\Omega |\nabla u_\delta|^2 dxdt \\&= \int_0^T \int_\Omega \frac{|\nabla u_\delta|^2}{u_\delta}u_\delta dxdt \\ &\leq  \int_0^T \|u_\delta(t)\|_{L^\infty(\Omega)}\int_\Omega \frac{|\nabla u_\delta|^2}{u_\delta } dxdt \\ &\leq \|u_\delta\|_{L^\infty(0,T; L^\infty(\Omega))} \int_0^T \int_\Omega \frac{|\nabla u_\delta|^2}{u_\delta} dxdt \\ 
		&\leq M
	\end{align*}
	with a constant $M$ is independent to $\delta$. By a similar argument, we obtain the conclusion for the gradient of $v_\delta$
	\[
	\|\nabla_\Gamma v_\delta\|^2_{L^2(0,T;L^2(\Gamma)} \leq M,
	\]
	which ends the proof.
\end{proof}
Finally, we need show that the time-derivative of $(u_\delta,v_\delta)$ is also bounded in a suitable functional space. 
\begin{bd}\label{lm:time_deri_approx}
	It holds that $\partial_t u_\delta$ and $\partial_tv_\delta$ are bounded in the space $L^2(0,T;H^{1}(\Omega)^*)$ and $L^2(0,T;H^{1}(\Gamma)^*)$, respectively, uniformly in $\delta>0$.
\end{bd}
\begin{proof}
	Consider the weak formulation for $u_\delta$, it holds for all $\varphi \in L^2(0,T;H^1(\Omega))$
	\begin{equation}\label{wfu1_approx}
		\int_0^T\langle \partial_t u_\delta, \varphi \rangle_\Omega +\int_\Omega  d_u \nabla u_\delta \nabla\varphi dxdt = - \frac{\alpha}{\epsilon} \int_0^T\int_\Gamma F_\delta(u_\delta,v_\delta)\varphi dSdt. 
	\end{equation}
	The right hand side can be estimated by using the fact that $|F(u_\delta,v_\delta)| \leq |u_\delta^\alpha - v_\delta^\beta|$ and $\{u_\delta\}, \{v_\delta\}$ are uniformly bounded in $L^\infty(0,T;L^\infty(\Gamma))$
	\begin{equation*}
		\frac{\alpha}{\epsilon} \int_0^T\int_\Gamma F_\delta(u_\delta,v_\delta)\varphi dSdt \leq C\|\varphi\|_{L^2(0,T;L^2(\Gamma))} \leq C\|\varphi\|_{L^2(0,T;H^1(\Omega))}.
	\end{equation*}
	Then, combining this with Lemma \ref{lm:gradient_approx}, we have the estimation
	\begin{equation*}
		\int_0^T \langle \partial_t u_\delta, \varphi\rangle_\Omega  dt\leq C\|\varphi\|_{L^2(0,T;H^1(\Omega))}
	\end{equation*}
	for all $\varphi \in L^2(0,T;H^1(\Omega))$, which gives $\{\partial_t u_\delta\}$ is bounded uniformly in $L^2(0,T;H^{1}(\Omega)^*)$. Using similar argument, we also have the boundedness for $\{\partial_t v_\delta\}$.
\end{proof}

Now, we can pass the limit of the sequence $\{u_\delta,v_\delta\}$ as $\delta \to 0$, to obtain the solution of \eqref{sys}.
\begin{proof}[Proof of the Theorem \ref{exist_sol}]
	From the boundedness in Lemmas \ref{lm:gradient_approx} and \ref{lm:time_deri_approx}, we can appply Aubin-Lions lemma to get the strong convergence as $\delta \to 0$, up to a subsequece,
	\begin{equation*}
		u_{\delta} \to u \text{ in } L^2(0,T;L^2(\Omega)) \quad \text{ and } \quad v_{\delta} \to u \text{ in } L^2(0,T;L^2(\Gamma))
	\end{equation*}
	for some non-negative functions $u\in L^2(0,T;H^1(\Omega))\cap L^{\infty}(0,T;L^{\infty}(\Omega))$ and $v\in L^2(0,T;H^1(\Gamma))\cap L^{\infty}(0,T;L^{\infty}(\Gamma))$. Thanks to the uniform-in-$\delta$ boundedness of $u_\delta$ and $v_\delta$, we also obtain $F_\delta(u,v) \to u^\alpha-v^\beta$ as $\delta \to 0$ in $L^2(0,T;L^2(\Omega)\times L^2(\Gamma))$. Therefore, we can pass to the limit as $\delta \to 0$ in the weak formulation of \eqref{approximation_system} to get that $(u,v)$ is a weak bounded solution to \eqref{sys}. Furthermore, thanks to the boundedness of $u,v$ and the local Lipschitz continuity of the nonlinearities, the uniqueness follows.
\end{proof}
\subsection{Convergence of the sequence of solution}
From now on, we use the subscript and denote by $(\ue,\ve)$ the solution to \eqref{sys} to highlight the dependence on $\epsilon>0$.
The main result of this part is the following theorem.
\begin{dl}\label{thm:converge}
	Let $(\ue,\ve)$ be the weak solution of system \eqref{sys} with parameter $\epsilon = 1/k$ and non-negative initial value $(u_0, v_0)$ $\in L^\infty(\Omega) \times L^\infty(\Gamma)$. Then, as $\epsilon \to 0$, the sequence $\{(\ue,\ve)\}$ has a subsequence that converges strongly in $L^2(0,T;$ $L^2(\Omega) \times L^2(\Gamma))$ to $(w,z)$, where $z = (w|_\Gamma)^{\alpha/\beta}$ and $w$ is a weak solution to \eqref{limit_system}.
\end{dl}
The proof of this theorem uses some intermediate lemmas, which we show in the subsequent part.
\begin{bd}\label{lm:ub}
	Let $(\ue, \ve)$ be the weak solution of the system \eqref{sys}, correspond to parameter $\epsilon$ and the non-negative initial condition $(u_0, v_0) \in L^\infty(\Omega) \times L^\infty(\Gamma)$. We have the upper bound
	\begin{equation*}
		\|\ue\|_{L^\infty(0,T;L^\infty(\Omega))} + \|\ve\|_{L^\infty(0,T;L^\infty(\Gamma))} \leq M,
	\end{equation*}
	where $M$ is a constant that does not depend on $\epsilon$.
\end{bd}
\begin{proof}
	The proof is similar to the Lemma \ref{lm:ub_approx}, with the remark that the constant $M$ in therein is also independent of $\epsilon>0$. Therefore, we only give a sketch here. By considering energy functions, for $p\in \mathbb N \backslash\{0\}$,
	\begin{equation}\label{entropy_func1}
		E_p[\ue,\ve](t) := \frac{1}{p\alpha^2 + \alpha}\int_\Omega \ue^{p\alpha + 1} dx + \frac{1}{p\beta^2 + \beta}\int_\Gamma \ve^{p\beta + 1} dS,
	\end{equation}
	we can use the weak formulation of $(\ue,\ve)$ to obtain, a.e. $t \in [0,T]$,
	\begin{equation}\label{est1:H_t<0}
		E_p(t) \leq E_p(0).
	\end{equation}
	This implies 
	\begin{align}\label{ieq2.2.1}
		\frac{1}{p\alpha^2 + \alpha}\int_\Omega \ue^{p\alpha + 1}(t) dx + &\frac{1}{p\beta^2 + \beta}\int_\Gamma \ve^{p\beta + 1}(t) dS \nonumber \\ &\leq \frac{1}{p\alpha^2 + \alpha}\int_\Omega u_0^{p\alpha + 1} dx + \frac{1}{p\beta^2 + \beta}\int_\Gamma v_0^{p\beta + 1} dS.
	\end{align}
	Taking the $p$-root and let $p\to \infty$, we obtain
	\begin{equation*}
		\|\ue(t)\|_{L^{\infty}(\Omega)} + \|\ve(t)\|_{L^{\infty}(\Gamma)} \le C
	\end{equation*}
	for a constant $C$ indepedent of $\epsilon>0$.
\end{proof}
We also need some estimates of the gradient, which we utilize the entropic structure.
\begin{bd}\label{lm:gradient_bound}
	We have the estimate:
	\begin{equation*}
		\|\nabla \ue\|_{L^2(0,T,L^2(\Omega))} +  \|\nabla_\Gamma \ve\|_{L^2(0,T,L^2(\Gamma))} \leq C
	\end{equation*}
	with a constant $C$ does not depend on $\epsilon$.
\end{bd}
\begin{proof}
	To prove this lemma, we introduce a modified logarithm entropy function, which is similar to Lemma \ref{lm:gradient_approx}
	\begin{equation}\label{log_entropy2}
		E_\lambda[\ue,\ve](t) = \int_\Omega (\ue + \lambda)(\log (\ue+\lambda) - 1)dx \nonumber + \int_\Gamma (\ve + \lambda)(\log (\ve + \lambda) - 1)dS.
	\end{equation}
	Direct computations give us the identity
	\begin{align*}
		\int_0^T \left( d_u\int_\Omega \frac{|\nabla \ue|^2 }{\ue + \lambda}dx + d_v\int_\Gamma \frac{|\nabla_\Gamma \ve|^2}{\ve + \lambda}dS \right)dt  = E_\lambda(0) - E_\lambda(T) + \frac{1}{\epsilon}\int_0 ^T \int_\Gamma (\ve ^\beta - \ue^\alpha)\log\frac{(\ue + \lambda)^\alpha}{(\ve + \lambda)^\beta}dS.
	\end{align*}
	Using the fact that $E_\lambda(0)$ and $E_\lambda(T)$ are bounded, we obtain the estimation
	\begin{equation}\label{est2.3.1}
		d_u\int_0^T \int_\Omega \frac{|\nabla \ue|^2 }{\ue + \lambda}dxdt + d_v\int_0^T\int_\Gamma \frac{|\nabla_\Gamma \ve|^2}{\ve + \lambda}dSdt \nonumber  \leq C + \frac{1}{\epsilon}\int_0 ^T \int_\Gamma (\ve ^\beta - \ue^\alpha)\log\frac{(\ue + \lambda)^\alpha}{(\ve + \lambda)^\beta}dS.
	\end{equation}
	Now, we pass the limit $\lambda \to 0^+$ for the estimation \eqref{est2.3.1}, similar to Lemma \ref{lm:gradient_approx}.
	\begin{equation*}
		d_u\int_0^T\int_\Omega \frac{|\nabla \ue|^2 }{\ue}dxdt +
		d_v\int_0^T\int_\Gamma \frac{|\nabla_\Gamma \ve|^2}{\ve}dSdt\leq C,
	\end{equation*}
	where $C$ is a constant that does not depend on $\epsilon$.
	\indent Finally, from the Lemma \ref{lm:ub}, we can show that the gradient is also bounded uniformly 
	\begin{align*}
		\|\nabla u_\epsilon\|^2_{L^2(0,T;L^2(\Omega))} \leq  \int_0^T \|u_\epsilon(t)\|_{L^\infty(\Omega)}\int_\Omega \frac{|\nabla u_\epsilon|^2}{u_\epsilon } dxdt  \leq C
	\end{align*}
	with a constant $C$ is independent to $\epsilon$. By a similar argument, we obtain the conclusion for the gradient of $\ve$
	\[
	\|\nabla_\Gamma v_\epsilon\|^2_{L^2(0,T;L^2(\Gamma)} \leq C.
	\]
\end{proof}

In order to the apply the Aubin-Lions lemma, we need uniform bound of the time derivative. However, we see that testing each equation of $\ue$ and $\ve$ will not help due to the singular reaction term as $\epsilon\to 0$. In order to overcome this issue, we look at the time derivative of $(\ue,\ve)$ in a suitable product space, which allows us to eliminate the singular reaction term when testing the system with test functions.
\begin{bd}\label{lm:time_deri_bound}
	There exists a functional space $Z$, which is compactly embedded in $L^2(\Omega) \times L^2(\Gamma)$, such that we have $\{\partial_t (\ue,\ve)\}_{\epsilon>0}$ is bounded (uniformly with respect to $\epsilon$) in the space $L^2(0,T;Z^*)$.
\end{bd}

\begin{proof}
	We can rewrite the weak formulation as
	\begin{equation}\label{wfu1}
		\langle \partial_t \ue, \varphi \rangle_{\Omega} +\int_\Omega  d_u \nabla \ue \cdot \nabla\varphi dx = - \frac{\alpha}{\epsilon} \int_\Gamma (\ue^\alpha - \ve^\beta)\varphi dS,
	\end{equation}
	\begin{equation}\label{wfv1}
		\langle \partial_t \ve, \psi \rangle_{\Gamma} +\int_\Gamma  d_v \nabla_\Gamma \ve \cdot \nabla_\Gamma \psi dS = \frac{\beta}{\epsilon}\int_\Gamma (\ue^\alpha - \ve^\beta)\psi dS.
	\end{equation}
	By adding the right hand sides of these equations and choose an appropriate test functions, we can remove the paramter $\epsilon$. Indeed, we define a new functional space for the test function (for more detail about this functional space, see, e.g. \cite{VAZQUEZ20112143}).
	\begin{equation}
		Z := \{(\phi,\phi|_{\Gamma}): \phi \in H^1(\Omega) \text{ and } \phi|_{\Gamma} \in H^1(\Gamma)\},
	\end{equation}
	and inherited the norm and inner product of $H^1(\Omega) \times H^1(\Gamma)$
	\begin{equation*}
		\|(\phi,\phi|_{\Gamma})\|_Z := \|\phi\|_{H^1(\Omega)} + \|\phi|_{\Gamma}\|_{H^1(\Gamma)},
	\end{equation*}
	\begin{equation*}
		\Bigl((\phi_1,\phi_1|_{\Gamma}), (\phi_2,\phi_2|_{\Gamma})\Bigr)_Z := (\phi_1,\phi_2)_{H^1(\Omega)} + (\phi_1|_{\Gamma}, \phi_2|_{\Gamma})_{H^1(\Gamma)}.
	\end{equation*}
	To show that the time derivative of $(\ue,\ve)$ is bounded in $L^2(0,T;Z^*)$, we take $(\varphi,\varphi |_{\Gamma}) \in Z$ arbitrary and multiply \eqref{wfu1} by $\beta \varphi$, \eqref{wfv1} by $\alpha \varphi$ and take the sum of these and take the integration from $0$ to $T$, we have:
	\begin{align}\label{eq2.4.1}
		\int_0^T\bigg(\int_\Omega \alpha \langle \partial_t \ue, \varphi \rangle_{\Omega}dx &+ \int_\Gamma \beta \langle \partial_t \ve, \varphi |_{\Gamma} \rangle_{\Gamma}dS\bigg)dt \nonumber \\ = -\int_0^T&\bigg(\int_\Omega \beta d_u \nabla \ue \cdot \nabla \varphi dx + \int_\Gamma \alpha d_v \nabla_\Gamma \ve \cdot \nabla_\Gamma \varphi |_{\Gamma} dS\bigg)dt.
	\end{align}
	Setting $V := L^2(0,T;Z)$, the dual of $V$ is $V^*$, which can be identified by $L^2(0,T;Z^*)$. So, the left hand side of equation \eqref{eq2.4.1} can be rewritten as the pairing between $V^*$ and $V$:
	\begin{equation}
		\label{eq2.4.2}
		\int_0^T\bigg(\int_\Omega \alpha \langle \partial_t \ue, \varphi \rangle_{\Omega}dx + \int_\Gamma \beta \langle \partial_t \ve, \varphi|_{\Gamma}
		\rangle_{\Gamma}dS\bigg)dt   = \langle \partial_t(\alpha\ue,\beta\ve); (\varphi,\varphi |_{\Gamma}) \rangle_{V^*, V}.
	\end{equation}
	Moreover, the right hand side of \eqref{eq2.4.1} can be bounded in the norm of $V$ due to the Lemma \ref{lm:gradient_bound}, which give the estimate:
	\begin{align*}
		&\int_0^T\int_\Omega \beta d_u \nabla \ue \cdot \nabla \varphi dx + \int_\Gamma \alpha d_v \nabla_\Gamma \ve \cdot \nabla_\Gamma (\varphi |_{\Gamma}) dSdt \\ &\leq \, \beta d_u \|\nabla \ue\|_{L^2(0,T;L^2(\Omega))}\|\nabla \varphi\|_{L^2(0,T;L^2(\Omega))} + \alpha d_v \|\nabla_\Gamma \ve\|_{L^2(0,T;L^2(\Gamma))}\|\nabla_\Gamma (\varphi |_{\Gamma})\|_{L^2(0,T;L^2(\Gamma))} \\  &\leq \,C(\|\varphi\|_{L^2(0,T;H^1(\Omega))} + \|\varphi |_{\Gamma}\|_{L^2(0,T;H^1(\Gamma))})\\  &\leq  C\, \|(\varphi, \varphi |_{\Gamma})\|_{V},
	\end{align*}
	where the first inequality given by Holder inequality. In the estimate, $C$ is a constant that does not depend on $\epsilon$. Combine with \eqref{eq2.4.1} and \eqref{eq2.4.2}, we have:
	\begin{equation}
		\langle \partial_t(\alpha\ue,\beta\ve); (\varphi,\varphi |_{\Gamma}) \rangle_{V^*, V} \leq C \|(\varphi, \varphi |_{\Gamma})\|_{V}
	\end{equation}
	for all $(\varphi,\varphi |_{\Gamma}) \in V$, or equivalently
	\begin{equation*}
		\|\partial_t(\ue,\ve)\|_{V^*} \leq C .
	\end{equation*}
	Hence, we have $Z$ is the required functional space.
\end{proof}
Before proving the main theorem, we present a definition for weak solutions to the limit system \eqref{limit_system}.
\begin{dn}\label{def_weak_sol_lim}
	We call $w$ a weak solution of problem \eqref{limit_system} if it satisfies the regularity
	\begin{equation*}
		w \in C([0,T]; L^2(\Omega)) \text{ and } w \in L^\infty(0,T;L^\infty(\Omega)) \cap L^2(0,T;H^1(\Omega)),
	\end{equation*}
	\begin{equation*}
		w|_{\Gamma} \in C([0,T]; L^2(\Gamma)) \text{ and } w|_{\Gamma}^{\alpha/\beta} \in L^\infty(0,T;L^\infty(\Gamma)) \cap L^2(0,T;H^1(\Gamma)),
	\end{equation*}
	and it satisfies the following weak formulation
	\begin{align}\label{wf3.1}
		&\beta\int_0^T\int_\Omega -w\varphi_t +\nabla w \cdot \nabla \varphi dx dt  + \alpha\int_0^T \int_\Gamma -(w|_{\Gamma})^{\alpha/\beta}(\varphi |_{\Gamma})_t +\nabla_\Gamma (w|_{\Gamma}^{\alpha/\beta}) \cdot \nabla_\Gamma (\varphi |_{\Gamma}) dSdt \\ &= \beta\int_\Omega u_0\varphi(0)dx + \alpha\int_\Gamma v_0\varphi |_{\Gamma}(0)dS \nonumber
	\end{align}
	where the test function $(\varphi,\varphi |_{\Gamma}) \in C^1([0,T],H^1(\Omega) \times H^1(\Gamma))$ and $\varphi(T) = 0$ (which also implies $\varphi |_{\Gamma}(T) = 0$). 
\end{dn}
\begin{proof}[Proof of Theorem \ref{thm:converge}]
	First, from Lemmas \ref{lm:gradient_bound} and \ref{lm:time_deri_bound}, we have the following strong convergence as $\epsilon \to 0$, up to a subsequence
	\[
	\ue \to w \text{ in } L^2(0,T;L^2(\Omega)) \quad \text{ and } \quad \ve \to z \text{ in } L^2(0,T;L^2(\Gamma)).
	\]
	Thanks to the $L^\infty$-boundedness, we have in fact
	\begin{equation}\label{Lp-convergence}
		\ue \to w \text{ in } L^p(0,T;L^p(\Omega)) \quad \text{ and } \quad \ve \to z \text{ in } L^p(0,T;L^p(\Gamma))
	\end{equation}
	for any $1<p<\infty$. We now show the identity $w^\alpha |_{\Gamma}(t) = z^\beta(t)$ by starting again with the entropy function in Lemma \ref{lm:ub}, choose $p = 1$ and take the integration from $0$ to $T$, the uniform boundedness of $(\ue,\ve)$ implies
	\begin{equation*}
		\frac{1}{\epsilon}\|\ue^\alpha - \ve^\beta\|^2_{L^2(\Gamma \times (0,T))} \leq C.
	\end{equation*}
	where $C$ is independent of $\epsilon$, which means $\ue^\alpha - \ve^\beta \to 0$ strongly in $L^2(\Gamma \times (0,T))$. Then, we have the weak convergence $\ue^\alpha - \ve^\beta \to 0$ in the same functional space, or in other words
	\begin{equation*}
		\lim_{\epsilon \to 0} \int_0^T\int_\Gamma (\ue^\alpha - \ve^\beta)\eta dS= 0
	\end{equation*}
	for all $\eta \in L^2(\Gamma \times (0,T))$. 
	On the other hand, thanks to Lemma \ref{lm:gradient_bound} and \eqref{Lp-convergence}, we have the weak convergence,
	\begin{equation*}
		\begin{cases}
			\ue^\alpha \rightharpoonup w^\alpha \in L^2(0,T;H^1(\Omega)), \\ 
			\ve^\beta \rightharpoonup z^\beta \in L^2(0,T;H^1(\Gamma)). \\ 
		\end{cases}
	\end{equation*}
	Due to trace Theorem, it implies
	\begin{equation*}
		\begin{cases}
			\uei^\alpha \rightharpoonup w^\alpha \in L^2(0,T;L^2(\Gamma)), \\ 
			\vei^\beta \rightharpoonup z^\beta \in L^2(0,T;L^2(\Gamma)). \\ 
		\end{cases}
	\end{equation*}
	So, for any $\eta \in L^2(0,T;L^2(\Gamma))$
	\begin{equation*}
		\begin{cases}
			\lim_{i \to \infty} (\uei^\alpha - w^\alpha,\eta)_{L^2(0,T;L^2(\Gamma))} = 0, \\ 
			\lim_{i \to \infty} (\vei^\beta - z^\beta,\eta)_{L^2(0,T;L^2(\Gamma))} = 0. \\ 
		\end{cases}
	\end{equation*}
	We can then deduce
	\begin{equation*}
		(w^\alpha - z^\beta, \eta)_{L^2(0,T;L^2(\Gamma))} = 0
	\end{equation*}
	for all $\eta \in L^2(0,T;L^2(\Gamma))$, which give the conclusion
	\begin{equation*}
		w^\alpha = z^\beta 
	\end{equation*}
	for a.e on $\Gamma \times (0,T)$. So, we can rewrite the limit $(w,z)$ by $(w,(w|_{\Gamma})^{\alpha/\beta})$. 		\indent Next, we will complete the proof of theorem by showing that $w$ satisfies the weak formulation \eqref{wf3.1}. First, we recall the weak formulation in the Introduction part, and choose the test function $(\varphi, \varphi|_{\Gamma}) \in C([0,T];Z)$
	\begin{equation}
		\int_0^T \int_\Omega (-\ue\varphi_t + d_u \nabla \ue \cdot \nabla \varphi)dxdt = \int_\Omega u_0\varphi(0)dx  - \frac{\alpha}{\epsilon} \int_0^T\int_\Gamma (\ue^\alpha - \ve^\beta)\varphi dSdt \label{wfu3}
	\end{equation}
	\begin{equation}\label{wfv3}
		\int_0^T \int_\Gamma (-\ve(\varphi|_{\Gamma})_t + d_v \nabla_\Gamma \ve \cdot \nabla_\Gamma (\varphi|_{\Gamma}))dSdt = \int_\Gamma v_0 \varphi|_{\Gamma}(0)dS + \frac{\beta}{\epsilon} \int_0^T\int_\Gamma (\ue^\alpha - \ve^\beta)\varphi|_{\Gamma} dSdt 
	\end{equation}
	Then, multiply \eqref{wfu3} by $\beta$, \eqref{wfv3} by $\alpha$, and take the sum, we obtain
	\begin{align}\label{eq:3.1.2}
		\int_0^T \int_\Omega -\ue\varphi_t + d_u \nabla \ue \cdot \nabla \varphi dxdt &+ \int_0^T \int_\Gamma -\ve(\varphi|_\Gamma)_t + d_v \nabla_\Gamma \ve \cdot \nabla_\Gamma (\varphi|_\Gamma)dSdt \nonumber\\  &= \int_\Omega u_0\varphi(0)dx + \int_\Gamma v_0\varphi|_\Gamma(0)dS.
	\end{align}
	Letting $\epsilon\to 0$ gives the weak formulation \eqref{wf3.1}, which concludes this proof.
\end{proof}
\section{Convergence rate}
In this section, we will investigate the convergence rate when $\alpha = \beta$, using the ideas from \cite{iida2006diffusion}. For the convergence rate, we assume more regularity of the initial condition 
\begin{equation*}
	u_0 \in H^1(\Omega) \, , v_0 \in H^1(\Gamma) \text{ and } v_0= u_0|_{\Gamma}.
\end{equation*}
This kind of condition is called compatibility condition for the heat equation with Wentzell boundary condition
\begin{equation}\label{lim_eq_linear}
	\begin{cases}
		\partial_t w - d_u\Delta w = 0, &x\in\Omega, t>0\\
		d_u \nabla w \cdot {\mathbf n} = -\partial_t w + d_v\Delta_{\Gamma}w, &x\in\Gamma,t>0\\
		w(x,0) = u_0(x) , &x\in\Omega,t>0\\
		w|_{\Gamma}(x,0) = v_0(x)  , &x\in\Gamma,t>0.\\
	\end{cases}
\end{equation}
Remark that the solution of \eqref{lim_eq_linear} $w$ satisfies the following weak formulation
\begin{align}\label{wfw3}
	\langle w_t, \varphi \rangle_\Omega + \langle (w|_{\Gamma})_t, \varphi|_{\Gamma} \rangle_\Gamma  + d_u \int_\Omega \nabla w \cdot \nabla \varphi dx + d_v \int_\Gamma \nabla_\Gamma (w|_{\Gamma})\cdot \nabla_\Gamma (\varphi |_{\Gamma}) dS = 0 
\end{align}
for all $\varphi \in H^1(\Omega)$ with $\varphi|_{\Gamma} \in H^1(\Gamma)$ and a.e. $t \in (0,T)$.
The existence of a unique global solution has been obtained in \cite{VAZQUEZ20112143}. The solution of this problem enjoys the further regularity in the following lemma, which is needed for our analysis.
\begin{bd}
	The solution of the system \eqref{lim_eq_linear} $w$ has the following property
	\begin{equation*}
		\partial_t(w,w_{|\Gamma}) \in L^2(0,T,L^2(\Omega) \times L^2(\Gamma)).
	\end{equation*}
\end{bd}
\begin{proof}
	We recall the variational formulation of \eqref{lim_eq_linear}, which is the case $\alpha = \beta$ of the formula \eqref{wf3.1}
	\begin{align*}\label{eq:3.1.3}
		\int_0^T \int_\Omega -w\varphi_t + d_u \nabla w \cdot \nabla \varphi dxdt &+ \int_0^T \int_\Gamma -w_{|\Gamma}(\varphi|_\Gamma)_t + d_v \nabla_\Gamma w_{|\Gamma} \cdot \nabla_\Gamma (\varphi|_\Gamma)dSdt \nonumber\\  &= \int_\Omega u_0\varphi(0)dx + \int_\Gamma v_0\varphi|_\Gamma(0)dS,
	\end{align*}
	which can be rewritten to
	\begin{equation}\label{wf3.3}
		\frac{d}{dt}\langle (w,w_{|\Gamma}),(\varphi, \varphi_{|\Gamma}) \rangle + a(w, w_{|\Gamma}, \varphi, \varphi_{|\Gamma}) = 0.
	\end{equation}
	In the second equation, we define the bilinear form
	\begin{equation*}
		a(w, w_{|\Gamma}, \varphi, \varphi_{|\Gamma}) := \int_\Omega  d_u \nabla w \cdot \nabla \varphi dx + \int_\Gamma d_v \nabla_\Gamma w_{|\Gamma} \cdot \nabla_\Gamma (\varphi|_\Gamma)dS
	\end{equation*}
	and the dual is in $L^2$ (which actually is restricted on $Z^* \times Z$). We already have the unique existence of $(w,w_{|\Gamma})$ in the functional space
	\begin{equation*}
		(w,w_{|\Gamma}) \in L^2(0,T,Z)\cap C([0,T]; L^2(\Omega) \times L^2(\Gamma)) \quad \partial_t (w,w_{|\Gamma}) \in L^2(0,T,Z^*)
	\end{equation*}
	To prove the regularity for $\partial_t (w,w_
	{|\Gamma})$, we use Garlekin approximation (see \cite{evans}[Theorem 7.1.5]). Indeed, we denote by $W_m = (w_m, (w_m)_{|\Gamma})$ the projection of $(w,w_{|\Gamma})$ from $Z$ on $Z_m$, the $m$-dimesional subspace. 
		
	Recall that the Garlekin's approximation use the smooth function, therefore we have $W_m \in C^1([0,T]; Z_m)$, or $\partial_t w_m(t) \in Z_m \subset Z$ for $t \in [0,T]$. Therefore, by choosing $W_m$ as the test function, we get
	\begin{equation*}
		\|\partial_t W_m\|_{L^2(\Omega) \times L^2(\Gamma)}^2 + a(W_m, \partial_t W_m) = 0.
	\end{equation*}
	Using the formulation that $\frac{d}{dt} a(W_m, W_m) = 2a(W_m, \partial_t W_m)$, we get
	\begin{equation*}
		2\|\partial_t W_m\|_{L^2(\Omega) \times L^2(\Gamma)}^2  + \frac{d}{dt} a(W_m, W_m) = 0
	\end{equation*}
	Then, integration in term of time, we have the uniform estimation in term of $m$
	\begin{equation*}
		\|\partial_t W_m\|_{L^2(0,T,L^2(\Omega) \times L^2(\Gamma))} \leq \frac 12 a(W_m(0), W_m(0)) \le C
	\end{equation*}
	where the last inequality is due to the regularity initial data. Then, by passing the limit $m \to \infty$, we have the conclusion
	\begin{equation*}
		\|\partial_t (w,w_{|\Gamma})\|_{L^2(0,T;L^2(\Omega) \times L^2(\Gamma))} \leq C.
	\end{equation*}
\end{proof}

	

In order to prove the convergence rate, we will need the following lemma showing positive lower bounds of solutions.
\begin{bd}\label{lm:lower_bound}
	If the inital data is stricly positive, $u_{0}(x) \ge m_0 > 0$ a.e. in $\Omega$ and $v_0(x) \ge m_0 > 0$ a.e. in $\Gamma$, then there exists $m > 0$ such that
	\begin{equation*}
		\ue(x,t) \geq m \; \text{ a.e. in } \Omega\times [0,\infty) \quad \text{ and } \quad  \ve(x,t) \ge m \; \text{ a.e. in } \Gamma \times [0,\infty).
	\end{equation*}
	Moreover, the trace of $\ue$ is also stricly positive
	\begin{equation*}
		\ue(x,t) \ge m \; \text{ a.e. in } \Gamma \times [0,\infty).
	\end{equation*}
\end{bd}
\begin{proof}
	We prove this lemma by showing that $1/\ue$ and $1/\ve$ are bounded from above. To do that, we introduce the following entropy function
	\begin{equation*}
		E_p[\ue,\ve](t) := \frac{1}{\alpha^2p -\alpha} \int_\Omega \frac{1}{\ue^{1-\alpha p}}dx + \frac{1}{\alpha^2p -\alpha} \int_\Omega \frac{1}{\ve^{1-\alpha p}}dS,
	\end{equation*}
	where parameter $p \in \mathbb{N} \backslash \{0\}$. Similarly, the time-derivative of the entropy funcion can be calculated by
	\begin{equation*}
		\frac{d}{dt}E_p = \frac{d_u}{\alpha}\int_\Omega \nabla \ue \cdot \nabla \frac{1}{\ue^{p\alpha}}dx + \frac{d_v}{\beta}\int_\Gamma \nabla \ve \cdot \nabla \frac{1}{\ve^{p\beta}}dS + \frac{1}{\epsilon}\int_\Gamma (\ue^\alpha - \ve^\beta)(\frac{1}{\ue^{\alpha p}} - \frac{1}{\ve^{\beta p}})dS.
	\end{equation*}
	We can check that these terms are non-positive, which give that
	\begin{equation*}
		E_p(0) \geq  E_p(t),
	\end{equation*}
	for all positive integer $p$ and a.e. time $t \in [0,T]$. Then, we have the estimation
	\begin{equation*}
		\frac{1}{\alpha^2p -\alpha} \int_\Omega \frac{1}{\ue(t)^{1-\alpha p}}dx + \frac{1}{\alpha^2p -\alpha} \int_\Omega \frac{1}{\ve(t)^{1-\alpha p}}dS \leq \frac{1}{\alpha^2p -\alpha} \int_\Omega \frac{1}{u_0^{1-\alpha p}}dx + \frac{1}{\alpha^2p -\alpha} \int_\Omega \frac{1}{v_0^{1-\alpha p}}dS.
	\end{equation*}
	Using similar argument to the end of proof in Lemma \ref{lm:ub_approx}, we can show that
	\begin{equation*}
		\Big\|\frac{1}{\ue}\Big\|_{L^{\infty}(0,T;L^\infty(\Omega))} + \Big\|\frac{1}{\ve}\Big\|_{L^{\infty}(0,T;L^\infty(\Gamma))} \leq C,
	\end{equation*}
	where $C$ is a constant that does not depend on $\epsilon$. Similar to Lemma \ref{lm: trace_esti}, we also have
	\begin{equation*}
		\Big\|\frac{1}{\ue}\Big\|_{L^{\infty}(0,T;L^\infty(\Gamma))} \leq C,
	\end{equation*}
	Fianlly, by choose $m = 1/C$, we get the conclusion.
\end{proof}
We are now ready to prove the convergence rate.
\begin{dl}[Convergence rate]
	Assume more that initial condition $(u_0,v_0) \in Z$ is strictly positive, the solution of \eqref{sys_linear} $\{(\ue,\ve)\}$ converges to the solution of $\eqref{lim_eq_linear}$ as $\epsilon \to 0$ with convergence rate
	\begin{equation*}
		\|u(t)-w(t)\|_{L^2(\Omega)} + \|v(t)- w|_{\Gamma}(t)\|_{L^2(\Gamma)} \leq c \sqrt{\epsilon}.
	\end{equation*}
\end{dl}
\begin{proof}
	When $\alpha = \beta$, the original system has the form
	\begin{equation}\label{sys_linear}
		\begin{cases}
			\partial_t\ue - d_u \Delta \ue = 0, &x\in\Omega, t>0,\\
			d_u \nabla \ue \cdot {\mathbf n} = -\frac{\alpha}{\epsilon}(\ue^\alpha -\ve^\alpha), &x\in\Gamma, t > 0,\\
			\partial_t \ve - d_v \Delta_{\Gamma}\ve = \frac{\alpha}{\epsilon}(\ue^\alpha - \ve^\alpha), &x\in\Gamma, t>0,\\
			\ue(x,0) = u_0(x) \geq 0, &x\in\Omega,\\
			\ve(x,0)= v_0(x) = (u_0)_{|\Gamma}(x) \geq 0, &x\in\Gamma ,
		\end{cases}
	\end{equation}
	and we use the following weak formulation
	\begin{equation}\label{wfu3.2}
		\langle \partial_t \ue, \varphi \rangle_{\Omega} + \int_\Omega d_u \nabla \ue \cdot \nabla\varphi dx = - \frac{\alpha}{\epsilon} \int_\Gamma (\ue^\alpha - \ve^\alpha)\varphi dS,
	\end{equation}
	\begin{equation}\label{wfv3.2}
		\langle \partial_t \ve, \psi \rangle_{\Gamma} + \int_\Gamma d_v \nabla_\Gamma \ve \cdot \nabla_\Gamma \psi dS = \frac{\alpha}{\epsilon}\int_\Gamma (\ue^\alpha - \ve^\alpha)\psi dS,
	\end{equation}
	for a.e. $t \in [0,T]$ and for all $(\varphi,\psi) \in H^1(\Omega) \times H^1(\Gamma)$.
	
	\medskip
	We are interested in the convergence rate for the system \eqref{sys_linear} converges to equation \eqref{lim_eq_linear} as $\epsilon \to 0$. First, fix a small $\epsilon$, we set
	\begin{align*}
		&U(x,t) := \ue(x,t) - w(x,t) \quad \text{on }\Omega \times [0,T],
		\\ & V(x,t) := \ve(x,t) - w(x,t) \quad \text{on }\Gamma \times [0,T] .
	\end{align*}
	Remark that we have $U-V = \ue-\ve$ a.e. on $\Gamma \times (0,T)$.
	Direct computations give
	\begin{align*}
		\frac{d}{dt}(\|U(t)\|^2_{L^2(\Omega)} + \|V(t)\|^2_{L^2(\Gamma)}) =& \langle \partial_t \ue,\ue-w \rangle_{\Omega} - \langle \partial_t w,\ue-w \rangle_{\Omega} + \langle \partial_t \ve,\ve-w \rangle_{\Gamma} - \langle \partial_t w,\ve-w \rangle_{\Gamma} \\ =& \langle \partial_t \ue,\ue-w \rangle_{\Omega} + \langle \partial_t \ve,\ve-w \rangle_{\Gamma} \\ &- \left(\langle \partial_t w,\ue-w \rangle_{\Omega} + \langle \partial_t w,\ue-w \rangle_{\Gamma}\right) +  \langle \partial_t w,\ue-\ve \rangle_{\Gamma}.
	\end{align*}
	Then, choosing the appropriate test function for weak formulation \eqref{wfw3}, \eqref{wfu3.2}  and \eqref{wfv3.2}, we obtain
	\begin{align*}
		\langle \partial_t \ue,\ue-w \rangle_{\Omega} = -d_u\int_\Omega \nabla \ue \cdot \nabla (\ue-w)dx- \frac{\alpha}{\epsilon} \int_\Gamma (\ue^\alpha - \ve^\alpha)(\ue-w) dS,
	\end{align*}
	\begin{align*}
		\langle \partial_t \ve,\ve-w \rangle_{\Gamma} = -d_v\int_\Gamma \nabla_\Gamma \ve \cdot \nabla_\Gamma (\ve-w)dS + \frac{\alpha}{\epsilon} \int_\Gamma (\ue^\alpha - \ve^\alpha)(\ve-w) dS,
	\end{align*}
	and
	\begin{align*}
		\langle \partial_t w,\ue-w \rangle_{\Omega} + \langle \partial_t w,\ue-w \rangle_{\Gamma} =  -d_u \int_\Omega \nabla w \cdot \nabla (\ue-w) dx  - d_v \int_\Gamma \nabla_\Gamma (w|_{\Gamma})\cdot \nabla_\Gamma (\ue-w) dS.
	\end{align*}
	Combine these, we have
	\begin{align*}
		\frac{d}{dt}(\|U(t)\|^2_{L^2(\Omega)} + \|V(t)\|^2_{L^2(\Gamma)}) 
		=& - d_u\|\nabla(\ue-w)\|_{L^2(\Omega)}^2  - d_v\|\nabla(\ve-w)\|_{L^2(\Gamma)}^2  \\ &- \frac{\alpha}{\epsilon} \int_\Gamma (\ue^\alpha - \ve^\alpha)(\ue-\ve)dS + \langle \partial_t w,\ue-\ve \rangle_{\Gamma}.
	\end{align*}
	The first and second terms of the right-hand side are non-positive, so we only need to work with the remaining term. For the integral term, by using mean value theorem we have
	\begin{align*}
		- \frac{\alpha}{\epsilon} \int_\Gamma (\ue^\alpha - \ve^\alpha)(\ue-\ve)dS =& - \frac{\alpha}{\epsilon} \int_\Gamma [(\xi(x)^{\alpha -1})(\ue- \ve)](\ue-\ve)dS \\
		=& - \frac{\alpha}{\epsilon} \int_\Gamma \xi(x)^{\alpha -1}|U-V|^2 dS,
	\end{align*}
	where $\xi(x) \in (\min\{u(x),v(x)\}, \max\{u(x),v(x)\})$ for a.e. $(x,t) \in \Gamma\times(0,T)$. Thanks to Lemma \ref{lm:lower_bound}, we arrive at
	\begin{equation*}
		- \frac{\alpha}{\epsilon} \int_\Gamma (\ue^\alpha - \ve^\alpha)(\ue-\ve)dS \leq \frac{-c_1}{\epsilon} \|U-V\|^2_{L^2(\Gamma)},
	\end{equation*}
	with $c_1$ is a constant that does not depend on $\epsilon$. On the other hand, the regularity of $w$ allows us to rewrite
	\begin{equation*}
		\langle \partial_t w,\ue-\ve \rangle_{\Gamma} = \left(\partial_t w, U-V \right)_{L^2(\Gamma)}.
	\end{equation*}
	From the regularity of $w$, we have $\|\partial_t w\|_{L^2(0,T;L^2(\Gamma))}\le c_2$ with $c_2$ being a constant independent of $\epsilon$. Combining the above, we deduce
	\begin{align*}
		\frac{d}{dt}(\|U(t)\|^2_{L^2(\Omega)} + \|V(t)\|^2_{L^2(\Gamma)}) &\leq \frac{-c_1}{\epsilon} \|U-V\|^2_{L^2(\Gamma)} + \|\partial_t w\|_{L^2(\Gamma)}\|U-V\|_{L^2(\Gamma)}.
	\end{align*}
	Fix  $t_0 \in [0,T]$, take the integration from $0$ to $t_0$ and using the fact that $U(0) = V(0) = 0$, we obtain the estimation
	\begin{align*}
		\|U(t_0)\|^2_{L^2(\Omega)} + \|V(t_0)\|^2_{L^2(\Gamma)} &\leq \int_0^{t_0} \frac{-c_1}{\epsilon} \|U-V\|^2_{L^2(\Gamma)}dt + c_2\|U - V\|_{L^2(0,t_0,L^2(\Gamma))} \\
		&\leq \frac{-c_1}{\epsilon} \|U-V\|^2_{L^2(0,t_0,L^2(\Gamma))} + c_2 \|U-V\|_{L^2(0,t_0,L^2(\Gamma))} \\
		&\leq \frac{-c_1}{2\epsilon} \|U-V\|^2_{L^2(0,t_0,L^2(\Gamma))} + c_3\epsilon \leq c_3 \epsilon
	\end{align*}
	which concludes the proof.
\end{proof}

\end{document}